\documentclass[12pt, reqno]{amsart}

\usepackage{amsfonts}

\usepackage{setspace}
\usepackage{graphicx}
\usepackage{eufrak}
\usepackage{mathrsfs}
\usepackage{yfonts}
\usepackage{hyperref}
\usepackage{graphicx}
\usepackage{amsmath, amsthm, amscd, amsfonts, amssymb, graphicx, color}
\usepackage{url}
\usepackage{enumerate}
\makeatletter
\let\reftagform@=\tagform@
\def\tagform@#1{\maketag@@@{(\ignorespaces\textcolor{blue}{#1}\unskip\@@italiccorr)}}
\renewcommand{\eqref}[1]{\textup{\reftagform@{\ref{#1}}}}
\makeatother
\usepackage{hyperref}
\hypersetup{colorlinks=true, linkcolor=red, anchorcolor=green,
    citecolor=cyan, urlcolor=red, filecolor=magenta, pdftoolbar=true}

\usepackage{epsfig}        
\usepackage{epic,eepic}       

\newtheorem{theorem}{Theorem}
\theoremstyle{plain}

\newtheorem{corollary}{Corollary}

\newtheorem{lemma}{Lemma}

\numberwithin{equation}{section}

\DeclareMathOperator{\spe}{sp}

\def\etal{et al.\,}

\begin{document}
    \title[Numerical radius inequalities  for Hilbert Space Operators ]
    {Refinements of Some Numerical radius inequalities   for Hilbert Space Operators }
    \author[M.W. Alomari]{Mohammad .W. Alomari}

    \address{Department of Mathematics, Faculty of Science and Information
        Technology, Irbid National University, 2600 Irbid 21110, Jordan.}
    \email{mwomath@gmail.com}

\date{\today}

\begin{abstract}
In this work, some generalizations and refinements   inequalities
for the numerical radius of the product of Hilbert space operators
are proved. New inequalities for the numerical radius of block
matrices of   Hilbert space operators are also established.
\end{abstract}

    \maketitle
    \section{Introduction}

  Let $\mathscr{B}\left( \mathscr{H}\right) $ be the Banach algebra
  of all bounded linear operators defined on a complex Hilbert space
  $\left( \mathscr{H};\left\langle \cdot ,\cdot \right\rangle
  \right)$  with the identity operator  $1_\mathscr{H}$ in
  $\mathscr{B}\left( \mathscr{H}\right) $.

  For a bounded linear operator $T$ on a Hilbert space
  $\mathscr{H}$, the numerical range $W\left(T\right)$ is the image
  of the unit sphere of $\mathscr{H}$ under the quadratic form $x\to
  \left\langle {Tx,x} \right\rangle$ associated with the operator.
  More precisely,
  \begin{align*}
  W\left( T \right) = \left\{ {\left\langle {Tx,x} \right\rangle :x
    \in \mathscr{H},\left\| x \right\| = 1} \right\}.
  \end{align*}
  Also, the numerical radius is defined to be
  \begin{align*}
  w\left( T \right) = \sup \left\{ {\left| \lambda\right|:\lambda
    \in W\left( T \right) } \right\} = \mathop {\sup }\limits_{\left\|
    x \right\| = 1} \left| {\left\langle {Tx,x} \right\rangle }
  \right|.
  \end{align*}

  The spectral radius of an operator $T$ is defined to be
  \begin{align*}
  r\left( T \right) = \sup \left\{ {\left| \lambda\right|:\lambda
    \in \spe\left( T \right) } \right\}.
  \end{align*}

  We recall that,  the usual operator norm of an operator $T$ is defined to be
  \begin{align*}
  \left\| T \right\| = \sup \left\{ {\left\| {Tx} \right\|:x \in H,\left\| x \right\| = 1} \right\},
  \end{align*}
  and
  \begin{align*}
  \ell \left( T \right): &= \inf \left\{ {\left\| {Tx} \right\|:x
    \in \mathscr{H},\left\| x \right\| = 1} \right\}
  \\
  &=      \inf \left\{ {\left|\left\langle {Tx,y} \right\rangle \right|:x,y \in
    \mathscr{H},\left\| x \right\| =\left\| y \right\|= 1} \right\}.
  \end{align*}
 It's well known that   the numerical radius is not submultiplicative, but it is satisfies  $w(TS)\le 4w\left(T\right) w\left(S\right)$ for all $T,S\in \mathscr{B}\left(  \mathscr{H}\right)$. In particular if $T,S$ are commute,  then $w(TS)\le 2w\left(T\right) w\left(S\right)$. Moreover, if $T,S$ are normal  then $w\left(\cdot\right)$ is submultiplicative $w(TS)\le w\left(T\right) w\left(S\right)$. Denote  $|T|=\left(T^*T\right)^{1/2}$  the absolute value of the operator $T$. Then we have $w\left(|T|\right) = \|T\|$. It's convenient to mention  that, the numerical radius norm is weakly unitarily invariant; i.e., $w\left(U^*TU\right) = w\left(T\right)$  for all unitary $U$. Also, let us  not miss the chance to mention the important property that  $w\left(T\right) = w\left(T^*\right)$  and $w\left(T^*T\right) = w\left(TT^*\right)$ for every $T\in \mathscr{B}\left( \mathscr{H}\right)$.

  A popular problem is the following: does the numerical radius of the product of operators commute, i.e., $w(TS)= w\left(ST\right)$  for any operators $T,S\in \mathscr{B}\left(\mathscr{H}\right)$?

 This problem has been given serious attention by many authors and in several resources (see \cite{G}, for example). Fortunately, it has been shown recently that,  for one of such  operators must be a multiple of a unitary operator, and we need  only to check $w\left(TS\right)=w\left(ST\right)$ for all rank one operators $S\in \mathscr{B}\left( \mathscr{H}\right)$ to arrive at the conclusion. This fact was proved by Chien \etal  in \cite{CGLTW}. For other related problems involving  numerical ranges and radius see \cite{CGLTW} and \cite{CKN}
 as well as the elegant work of Li \cite{LTWW} and the  references
 therein. For more classical and recent properties of  numerical range  and radius, see \cite{CGLTW} \cite{CKN}, \cite{LTWW} and the comprehensive books \cite{B},  \cite{H1} and \cite{H2}.

  On the other hand, it is well known that $w\left(\cdot\right)$ defines an operator norm on $\mathscr{B}\left( \mathscr{H}\right) $ which is equivalent to operator norm $\|\cdot\|$. Moreover, we have
  \begin{align}
  \frac{1}{2}\|T\|\le w\left(T\right) \le \|T\|\label{eq1.1}
  \end{align}
  for any $T\in \mathscr{B}\left( \mathscr{H}\right)$. The inequality is sharp.

  In 2003, Kittaneh \cite{FK1}  refined the right-hand side of \eqref{eq1.1}, where he proved that
  \begin{align}
  w\left(T\right) \le \frac{1}{2}\left(\|T\|+\|T^2\|^{1/2}\right)\label{eq1.2}
  \end{align}
  for any  $T\in \mathscr{B}\left( \mathscr{H}\right)$.

After that in 2005, the same author in \cite{FK} proved that
  \begin{align}
  \frac{1}{4}\|A^*A+AA^*\|\le  w^2\left(A\right) \le \frac{1}{2}\|A^*A+AA^*\|.\label{eq1.3}
  \end{align}
  The inequality is sharp. This inequality was also reformulated and generalized in \cite{EF} but in terms of Cartesian
decomposition.

  In 2007, Yamazaki \cite{Y} improved \eqref{eq1.1} by proving that
  \begin{align}
  w\left( T \right) \le \frac{1}{2}\left( {\left\| T \right\| + w\left( {\widetilde{T}} \right)} \right) \le \frac{1}{2}\left( {\left\| T \right\| + \left\| {T^2 } \right\|^{1/2} } \right),\label{eq1.4}
  \end{align}
  where $\widetilde{T}=|T|^{1/2}U|T|^{1/2}$ with unitary $U$.

  In 2008, Dragomir \cite{D4} used Buzano inequality to improve \eqref{eq1.1}, where he proved that
  \begin{align}
  w^2\left( T \right) \le \frac{1}{2}\left( {\left\| T \right\| + w\left( {T^2} \right)} \right). \label{eq1.5}
  \end{align}
  This result was also recently generalized by Sattari \etal in \cite{SMY}.

  This work, is divided into four sections, after this introduction, in Section \ref{sec2}, we recall some well-known inequalities for bounded linear operators. In Section \ref{sec3}, some generalizations and refinements of the numerical radius inequalities are proved. In Section \ref{sec4}, new refinement inequalities for the numerical radius of $n\times n$  Hilbert space operator   matrices    are established.

\section{Lemmas}\label{sec2}
In order to prove our results we need  a sequence of lemmas.
\begin{lemma}
\label{lemma1}  We have
    \begin{enumerate}
        \item The Power-Mean inequality
        \begin{align}
        a^\alpha  b^{1 - \alpha }  \le \alpha a + \left( {1 - \alpha } \right)b \le \left( {\alpha a^p  + \left( {1 - \alpha } \right)b^p } \right)^{\frac{1}{p}},   \label{PMI}
        \end{align}
        for all  $\alpha \in \left[0,1\right]$, $a,b\ge0$ and $ p\ge1$.

        \item The Power-Young inequality
        \begin{align}
        \label{YI}ab \le \frac{{a^{\alpha} }}{\alpha} + \frac{{b^{\beta} }}{\beta} \le \left( {\frac{{a^{p\alpha} }}{\alpha} + \frac{{b^{p\beta} }}{\beta}} \right)^{\frac{1}{p}}
        \end{align}
        for all $a,b\ge0$ and $\alpha,\beta>1$ with $\frac{1}{\alpha}+\frac{1}{\beta}=1$ and all $p\ge1$.
    \end{enumerate}
\end{lemma}
\begin{lemma} {\rm{(The McCarty inequality).}}
    \label{lemma2}  Let  $A\in \mathscr{B}\left( \mathscr{H}\right)^+ $, then
    \begin{align}
    \left\langle {Ax,x} \right\rangle ^p  \le \left\langle {A^p x,x} \right\rangle, \qquad p\ge1, \label{mc1}
    \end{align}
    for any unit vector $x\in\mathscr{H}$
\end{lemma}
The mixed Schwarz inequality was introduced in \cite{TK}, as
follows:
\begin{lemma}
    \label{lemma3}  Let  $A\in \mathscr{B}\left( \mathscr{H}\right)^+ $, then
    \begin{align}
    \left| {\left\langle {Ax,y} \right\rangle} \right|  ^2  \le \left\langle {\left| A \right|^{2\alpha } x,x} \right\rangle \left\langle {\left| {A^* } \right|^{2\left( {1 - \alpha } \right)} y,y} \right\rangle, \qquad 0\le \alpha \le 1. \label{eq2.4}
    \end{align}
    for any   vectors $x,y\in \mathscr{H}$
\end{lemma}

In order to generalize \eqref{eq2.4}, Kittaneh in \cite{FK4}  used
the following key lemma to prove a generalization of Kato's
inequality \eqref{eq2.4}.
\begin{lemma}
    \label{lemma4}Let  $A,B\in \mathscr{B}\left( \mathscr{H}\right)^+$. Then $\left[ {\begin{array}{*{20}c} A & {C^* }  \\C & B  \\\end{array}} \right]$  is positive in $\mathscr{B}\left(\mathscr{H}\oplus \mathscr{H}\right)$ if and only if $
    \left| {\left\langle {Cx,y} \right\rangle } \right|^2  \le \left\langle {Ax,x} \right\rangle \left\langle {By,y} \right\rangle $ for every vectors $x,y\in \mathscr{H}$.
\end{lemma}
Indeed, in \cite{FK4} we find that
\begin{lemma}
    \label{lemma5}  Let  $A,B\in \mathscr{B}\left( \mathscr{H}\right)$ such that $|A|B=B^*|A|$.
    If $f$ and $g$ are nonnegative continuous functions on $\left[0,\infty\right)$ satisfying $f(t)g(t) =t$ $(t\ge0)$, then
    \begin{align}
    \left| {\left\langle {ABx,y} \right\rangle } \right| \le r\left(B\right)\left\| {f\left( {\left| A \right|} \right)x} \right\|\left\| {g\left( {\left| {A^* } \right|} \right)y} \right\|\label{kittaneh.ineq}
    \end{align}
    for any   vectors $x,y\in  \mathscr{H} $.
\end{lemma}
Clearly, by setting  $B=1_{\mathscr{H}}$ and choosing
$f(t)=t^{\alpha}$, $g(t)=t^{1-\alpha}$, then the inequality
\eqref{kittaneh.ineq} reduces to \eqref{eq2.4}.

The following useful estimate  of a spectral radius was also
obtained by Kittaneh in \cite{F}.

\begin{lemma}
    \label{lemma6}If $A,B\in \mathscr{B}\left( \mathscr{H}\right)$. Then
\begin{multline}
r\left( {AB} \right)\label{fact3}\\\le \frac{1}{4}\left( {\left\|
{AB} \right\| + \left\| {BA} \right\| + \sqrt {\left(\left\| {AB}
\right\|- \left\| {BA} \right\|\right)^2 + 4m\left(A,B\right)}}
\right),
\end{multline}
where $m\left(A,B\right):=\min \left\{ {\left\| A \right\|\left\|
{BAB} \right\|,\left\| B \right\|\left\| {ABA} \right\|}
\right\}$.
\end{lemma}

In some of our results we need the following two fundamental norm
estimates, which  are:
 \begin{multline}
\label{fact1}\left\| {A+ B } \right\|\\\le \frac{1}{2}\left(
{\left\| A \right\| + \left\| B \right\| + \sqrt
    {\left( {\left\| A \right\| - \left\| B \right\|} \right)^2  +
        4\left\| {A^{1/2} B^{1/2} } \right\|^2 } } \right),
\end{multline}
and
\begin{align}
\label{fact2}\left\| {A^{1/2} B^{1/2} } \right\|  \le\left\| {A  B
} \right\| ^{1/2}.
\end{align}
 Both estimates are valid for all positive operators $A,B \in \mathscr{B}\left( \mathscr{H}\right)$. Also, it should be noted that \eqref{fact1} is sharper than the triangle inequality as pointed out by Kittaneh in \cite{FK3}. \\

A new refinement of Cauchy-Schwarz inequality was recently
obtained in   \cite{AM}, as follows:
\begin{lemma}
\label{lemma7}  Let $A\in \mathscr{B}\left( \mathscr{H}\right)^+$,
then
    \begin{align}
    \label{fact4}\left|\left\langle {Ax,y} \right\rangle\right|
    ^{2p}&\le   \left[\left\langle { A^p x,x} \right\rangle-
    \left\langle{\left| {A - \left\langle { A x,x} \right\rangle
            1_{\mathcal{H}} } \right|^px,x} \right\rangle
    \right]\\&\qquad\times \left[\left\langle { A^p y,y}
    \right\rangle- \left\langle{\left| {A - \left\langle
            { A y,y} \right\rangle 1_{\mathcal{H}} } \right|^py,y} \right\rangle \right]\nonumber \\
    &\le\left\langle { A^p x,x} \right\rangle\left\langle { A^p y,y}
    \right\rangle \nonumber
    \end{align}
    for all $p\ge2$ and every $x,y\in \mathcal{H}$.
\end{lemma}

\section{Numerical Radius Inequalities}\label{sec3}
Let us begin with the following result.
\begin{theorem}
    \label{thm1}    Let  $A,B\in \mathscr{B}\left( \mathscr{H}\right)$ such that $|A|B= B^*|A|$. If $f, g$ be
    nonnegative continuous functions on $\left[0,\infty\right)$
    satisfying $f(t)g(t) = t$, $(t \ge 0)$. Then
    \begin{align}
    w \left( AB \right) &\le \frac{1}{2} r \left(B \right)
    w\left( {\left(f^2\left( {\left| {A  } \right|} \right) +g^2\left( {\left| {A^*} \right|} \right)\right)  } \right) \nonumber
    \\
    &\le\frac{1}{8} \left( {\left\| B \right\| + \left\| {B^2 } \right\|^{1/2} } \right)\left\{ \left\| {f^2 \left( {\left| A \right|} \right)} \right\| + \left\| {g^2 \left( {\left| {A^* } \right|} \right)} \right\|  \right.\label{eq2.1}\\
    &\qquad\left.+  \sqrt {\left( {\left\| {f^2 \left( {\left| A \right|} \right)} \right\| - \left\| {g^2 \left( {\left| {A^* } \right|} \right)} \right\|} \right)^2  + 4\left\| {f\left( {\left| A \right|} \right)g\left( {\left| {A^* } \right|} \right)} \right\|^2 } \right\}.\nonumber
    \end{align}
    In particular, we have
    \begin{align*}
   w \left( A \right) &\le \frac{1}{2}
   w\left( {\left(f^2\left( {\left| {A  } \right|} \right) +g^2\left( {\left| {A^*} \right|} \right)\right)  } \right).
    \end{align*}
\end{theorem}
\begin{proof}
    Setting $y=x$ in  \eqref{kittaneh.ineq}, we get
    \begin{align*}
    \left| {\left\langle {ABx,x} \right\rangle } \right|  &\le r \left(B \right)\left\| {f\left( {\left| A  \right|} \right)x} \right\| \left\| {g\left( {\left| {A^* } \right|} \right)x} \right\|  \\
    &=  r \left(B \right)\left\langle {f^2\left( {\left| {A  } \right|} \right)x  ,x  } \right\rangle^{1/2}  \left\langle {g^2\left( {\left| {A^*} \right|} \right)x  ,x  } \right\rangle^{1/2}    \\
    &\le \frac{1}{2} r \left(B \right) \left(\left\langle {f^2\left( {\left| {A  } \right|} \right)x  ,x  } \right\rangle +  \left\langle {g^2\left( {\left| {A^*} \right|} \right)x  ,x  } \right\rangle  \right)\qquad \text{by\,\,\eqref{PMI}}\\
    &=   \frac{1}{2} r \left(B \right) \left\langle {\left(f^2\left( {\left| {A  } \right|} \right) +g^2\left( {\left| {A^*} \right|} \right)\right)x  ,x  } \right\rangle
    \end{align*}
    Thus,   by taking the supremum over $x\in \mathscr{H}$  we get that
  \begin{align*}
  \mathop {\sup }\limits_{\left\| x \right\| = 1}   \left| {\left\langle {ABx,x} \right\rangle } \right|&\le  \frac{1}{2} r \left(B \right)
  \mathop {\sup }\limits_{\left\| x \right\| = 1}
  \left\langle {\left(f^2\left( {\left| {A  } \right|} \right) +g^2\left( {\left| {A^*} \right|} \right)\right)x  ,x  } \right\rangle
  \\
  &=\frac{1}{2} r \left(B \right)
  w\left( {\left(f^2\left( {\left| {A  } \right|} \right) +g^2\left( {\left| {A^*} \right|} \right)\right)  } \right)
  \\
   & \left(\le \frac{1}{2} r \left(B \right) \left\| {\left(f^2\left( {\left| {A  } \right|} \right) +g^2\left( {\left| {A^*} \right|} \right)\right)} \right\|\right) \qquad\text{by\,\, \eqref{eq1.1}}.
    \end{align*}
    which proves the first inequality in \eqref{eq2.1}.
    \begin{align*}
    \mathop {\sup }\limits_{\left\| x \right\| = 1}   \left| {\left\langle {ABx,x} \right\rangle } \right|&\le \frac{1}{2} r \left(B \right)
    w\left( {\left(f^2\left( {\left| {A  } \right|} \right) +g^2\left( {\left| {A^*} \right|} \right)\right)  } \right)
    \\
    & \le \frac{1}{2} r \left(B \right) \left\| {\left(f^2\left( {\left| {A  } \right|} \right) +g^2\left( {\left| {A^*} \right|} \right)\right)} \right\|
    \end{align*}
The second inequality follows by employing the \eqref{fact1} on
the last inequality above i.e.,
\begin{align*}
\left\| {\left(f^2\left( {\left| {A  } \right|} \right) +g^2\left(
{\left| {A^*} \right|} \right)\right)} \right\| &\le \frac{1}{2}
r \left(B \right) \left\{ \left\| {f^2 \left( {\left| A \right|}
\right)} \right\| + \left\| {g^2 \left( {\left| {A^* } \right|}
\right)} \right\|  \right.
\\
&\qquad\left.+  \sqrt {\left( {\left\| {f^2 \left( {\left| A
\right|} \right)} \right\| - \left\| {g^2 \left( {\left| {A^* }
\right|} \right)} \right\|} \right)^2  + 4\left\| {f\left( {\left|
A \right|} \right)g\left( {\left| {A^* } \right|} \right)}
\right\|^2 } \right\}
    \end{align*}
 also, by using \eqref{fact3} with $A=1_{\mathscr{H}}$, we get
\begin{align*}
\mathop {\sup }\limits_{\left\| x \right\| = 1}   \left|
{\left\langle {ABx,x} \right\rangle } \right| &\le\frac{1}{8}
\left( {\left\| B \right\| + \left\| {B^2 } \right\|^{1/2} }
\right)\left\{ \left\| {f^2 \left( {\left| A \right|} \right)}
\right\| + \left\| {g^2 \left( {\left| {A^* } \right|} \right)}
\right\|  \right.
\\
&\qquad\left.+  \sqrt {\left( {\left\| {f^2 \left( {\left| A
\right|} \right)} \right\| - \left\| {g^2 \left( {\left| {A^* }
\right|} \right)} \right\|} \right)^2  + 4\left\| {f\left( {\left|
A \right|} \right)g\left( {\left| {A^* } \right|} \right)}
\right\|^2 } \right\}
\end{align*}
and this proves the second inequality in \eqref{eq2.1}.

 \end{proof}
\begin{corollary}
\label{cor1}    Let  $A,B\in \mathscr{B}\left( \mathscr{H}\right)
$ such that $|A|B= B^*|A|$. Then,
\begin{align}
 w \left( AB \right) &\le \frac{1}{2} r \left(B \right)
 w\left( {\left( \left| {A  } \right|^{2\alpha}  +\left| {A^*} \right|^{2\left(1-\alpha\right)}  \right)  } \right)\nonumber
\\
&\le\frac{1}{8} \left( {\left\| B \right\| + \left\| {B^2 } \right\|^{1/2} } \right)\left\{ \left\| { \left| A \right|^{2\alpha} } \right\| + \left\| { \left| {A^* } \right|^{2\left(1-\alpha\right)} } \right\|  \right.\label{eq3.2}\\
&\qquad\left.+  \sqrt {\left( {\left\| { \left| A \right|^{
2\alpha} } \right\| - \left\| { \left| {A^* }
\right|^{2\left(1-\alpha\right)} } \right\|} \right)^2  + 4\left\|
{ \left| A \right|^{\alpha} \left| {A^* } \right|^{1-\alpha} }
\right\|^2 } \right\}.\nonumber
\end{align}
for all $0\le \alpha \le 1$.
\end{corollary}
\begin{proof}
    Setting $f(t)=t^{\alpha}$ and $g(t)=t^{1-\alpha}$, $0\le\alpha \le 1$, $t\ge0$ in Theorem \ref{thm1}.
\end{proof}
\begin{corollary}
    \label{cor2}    Let  $A,B\in \mathscr{B}\left( \mathscr{H}\right)$ such that $|A|B= B^*|A|$.
In particular, we have
\begin{align}
\label{eq3.3} w\left( A B\right) \le   \frac{1}{4}  \left( \left\|
B \right\| + \left\| {B^2 } \right\|^{1/2} \right)\cdot \left(
\left\| A \right\| + \left\| {A^2 } \right\|^{1/2} \right)
\end{align}
\end{corollary}
\begin{proof}
 Setting $\alpha=\frac{1}{2}$ in \eqref{eq3.2} we get
\begin{align*}
w \left( AB \right) &\le \frac{1}{2} r \left(B \right) w\left(
{\left( \left| {A  } \right|^{2\alpha}  +\left| {A^*}
\right|^{2\left(1-\alpha\right)}  \right)  } \right)\nonumber
\\
&\le\frac{1}{8} \left( {\left\| B \right\| + \left\| {B^2 } \right\|^{1/2} } \right)\left\{ \left\| { \left| A \right| } \right\| + \left\| { \left| {A^* } \right| } \right\|  \right.\qquad (\text{by}\,\,\eqref{fact3}\,\, with\,\, A=1_\mathscr{H})\\
&\qquad\left.+  \sqrt {\left( {\left\| { \left| A \right| }
\right\| - \left\| { \left| {A^* }
            \right| } \right\|} \right)^2  + 4\left\|
    { \left| A \right|^{1/2} \left| {A^* } \right|^{1/2} }
    \right\|^2 } \right\}.\nonumber
\\
&=\frac{1}{4}  \left( \left\| B \right\| + \left\| {B^2 }
\right\|^{1/2} \right)\cdot \left( \left\| A \right\| + \left\|
{A^2 } \right\|^{1/2} \right)
\end{align*}
where the last inequality follows from \eqref{fact2}  and using
the  fact  that  $\||A|\|=\|A^*|\|=\|A\|$
 and this proves  the desired result.
\end{proof}
A generalization of Theorem \ref{thm1}  is given as follows:
\begin{theorem}
\label{thm2}   Let  $A,B\in \mathscr{B}\left( \mathscr{H}\right) $
such that $|A|B= B^*|A|$. If $f, g$ be nonnegative continuous
functions   on $\left[0,\infty\right)$  satisfying $f(t)g(t) = t$,
$(t \ge 0)$. Then,
\begin{align}
w^{p} \left( AB \right)  &\le r^p\left(B\right)\cdot w \left(  {
\frac{1}{\alpha }f^{\alpha p} \left( {\left| {A  } \right|}
\right) + \frac{1}{\beta }g^{\beta p} \left( {\left| {A^* }
\right|} \right)} \right)
\nonumber \\
 &\le r^p\left(B\right)\cdot
 \left\| {\frac{1}{\alpha }f^{\alpha p} \left( {\left| {A  } \right|} \right) + \frac{1}{\beta }g^{\beta p} \left( {\left| {A^* } \right|} \right)} \right\|
 \label{eq3.4}
\end{align}
for all $p\ge1$, $\alpha \ge \beta >1 $ with
$\frac{1}{\alpha}+\frac{1}{\beta}=1$ and $\beta p \ge2$. Moreover
we have
\begin{align}
\label{eq3.5}w^{p} \left( AB \right) &\le \frac{1}{2^{p+1}} \cdot
\gamma \cdot \left( \left\| B \right\| + \left\| {B^2 }
\right\|^{1/2} \right)^p
\\
&\times\left\{ {\left\| {f^{\alpha p} \left( {\left| {A  }
\right|} \right)} \right\| + \left\| {g^{\beta p} \left( {\left|
{A^* } \right|} \right)} \right\|+\Phi\left(f,g;A\right)}
\right\}, \nonumber
\end{align}
where $\Phi\left(f,g;A\right):= \sqrt{\left[ {\left| {f^{\alpha p}
\left( {\left| {A  } \right|} \right)} \right| - \left\| {g^{\beta
p} \left( {\left| {A^* } \right|} \right)} \right\|} \right]^2  +
4\left\| {f^{p\alpha } \left( {\left| {A  } \right|}
\right)g^{p\beta } \left( {\left| {A^* } \right|} \right)}
\right\| }$ and $\gamma=
\max\{\frac{1}{\alpha},\frac{1}{\beta}\}$.

\end{theorem}

\begin{proof}
Using the  mixed Schwarz inequality \eqref{kittaneh.ineq},  we
have
\begin{align*}
&\left| {\left\langle {ABx,x} \right\rangle } \right|^p\\ &\le r^p\left(B\right)\left\| {f\left( {\left| A \right|} \right)x} \right\|^p\left\| {g\left( {\left| {A^* } \right|} \right)x} \right\|^p \\
&=  r^p\left(B\right)\left\langle {f^2\left( {\left| {A } \right|} \right)x  ,x  } \right\rangle ^{\frac{p}{2}} \left\langle {g^2\left( {\left| {A^*} \right|} \right)x  ,x  } \right\rangle ^{\frac{p}{2}}  \\
&\le r^p\left(B\right) \left[\frac{1}{\alpha }\left\langle {f^2 \left( {\left| {A } \right|} \right)x  ,x  } \right\rangle ^{\frac{{\alpha p}}{2}}  + \frac{1}{\beta }\left\langle {g^2 \left( {\left| {A^* } \right|} \right)x  ,x  } \right\rangle ^{\frac{{\beta p}}{2}}  \right] \qquad \text{(by \eqref{YI})}\\
&\le  r^p\left(B\right)  \left[\frac{1}{\alpha }\left\langle {f^{\alpha p} \left( {\left| {A  } \right|} \right)x  ,x  } \right\rangle  + \frac{1}{\beta }\left\langle {g^{\beta p} \left( {\left| {A^* } \right|} \right)x  ,x  } \right\rangle \right] \qquad \text{(by \eqref{mc1})}\\
&=  r^p\left(B\right)  \left\langle {\left[
{\frac{1}{\alpha}f^{\alpha p} \left( {\left| {A   } \right|}
\right) +\frac{1}{\beta }g^{\beta p} \left( {\left| {A^* }\right|}
\right)} \right]x  ,x } \right\rangle.
\end{align*}
Taking the supremum over $x \in \mathscr{H}$, we obtain the first
inequality in \eqref{eq3.4}. To obtain the second inequality, by
utilizing \eqref{eq1.1} on the first inequality in  \eqref{eq3.4}
we have
\begin{align*}
&w \left(  { \frac{1}{\alpha }f^{\alpha p} \left( {\left| {A  }
\right|} \right) + \frac{1}{\beta }g^{\beta p} \left( {\left| {A^*
} \right|} \right)} \right)
\\
&\le\left\| {\frac{1}{\alpha }f^{\alpha p} \left( {\left| {A  }
\right|} \right) + \frac{1}{\beta }g^{\beta p} \left( {\left| {A^*
} \right|} \right)} \right\|
\\
&\le \max\{\frac{1}{\alpha},\frac{1}{\beta}\}\cdot \left\| {
f^{\alpha p} \left( {\left| {A  } \right|} \right) +  g^{\beta p}
\left( {\left| {A^* } \right|} \right)} \right\|
\\
&\le \frac{1}{2}\gamma\left( {\left\| {f^{\alpha p} \left( {\left| {A  } \right|} \right)} \right\| + \left\| {g^{\beta p} \left( {\left| {A^* } \right|} \right)} \right\|} \right. \qquad \qquad\qquad \qquad(\text{by\,\, \eqref{fact1} })\\
&\qquad\left. { + \sqrt {\left[ {\left| {f^{\alpha p} \left(
{\left| {A  } \right|} \right)} \right| - \left\| {g^{\beta p}
\left( {\left| {A^* } \right|} \right)} \right\|} \right]^2  +
4\left\| {f^{p\alpha/2 } \left( {\left| {A  } \right|}
\right)g^{p\beta/2 } \left( {\left| {A^* } \right|} \right)}
\right\|^2 } } \right)
\\
&\le \frac{1}{2}\gamma\left( {\left\| {f^{\alpha p} \left( {\left| {A  } \right|} \right)} \right\| + \left\| {g^{\beta p} \left( {\left| {A^* } \right|} \right)} \right\|} \right. \qquad \qquad \qquad\qquad (\text{by\,\, \eqref{fact2} })\\
&\qquad\left. { + \sqrt {\left[ {\left| {f^{\alpha p} \left(
                {\left| {A  } \right|} \right)} \right| - \left\| {g^{\beta p}
                \left( {\left| {A^* } \right|} \right)} \right\|} \right]^2  +
        4\left\| {f^{p\alpha } \left( {\left| {A  } \right|}
            \right)g^{p\beta } \left( {\left| {A^* } \right|} \right)}
        \right\| } } \right).
\end{align*}
Hence, by substituting all in  \eqref{eq3.4} we get
\begin{align*}
w^{p} \left( AB \right)  &\le r^p\left(B\right)\cdot w \left(  {
\frac{1}{\alpha }f^{\alpha p} \left( {\left| {A  } \right|}
\right) + \frac{1}{\beta }g^{\beta p} \left( {\left| {A^* }
\right|} \right)} \right)
\nonumber \\
&\le r^p\left(B\right)\cdot \left\| {\frac{1}{\alpha }f^{\alpha p}
\left( {\left| {A  } \right|} \right) + \frac{1}{\beta }g^{\beta
p} \left( {\left| {A^* } \right|} \right)} \right\|
\\
 &\le \frac{1}{2^{p+1}} \cdot
\gamma \cdot \left( \left\| B \right\| + \left\| {B^2 }
\right\|^{1/2} \right)^p
\\
&\qquad\times\left\{ {\left\| {f^{\alpha p} \left( {\left| {A  }
            \right|} \right)} \right\| + \left\| {g^{\beta p} \left( {\left|
            {A^* } \right|} \right)} \right\|+\Phi\left(f,g;A\right)}
\right\},
\end{align*}
where the last inequality follows from  \eqref{fact3} with
$A=1_{\mathscr{H}}$ and this is yield the required result, where
$\Phi\left(f,g;A\right)$ is defined above.
\end{proof}

A generalization of Sattari \etal inequality which was obtained in
\cite{SMY} is given as follows:
\begin{theorem}
\label{thm3}    Let  $A,B\in \mathscr{B}\left( \mathscr{H}\right)
$ such that
\begin{enumerate}
\item $AB=BA$, and

\item $|A^2|B^2=\left(B^2\right)^*|A^2|$.
\end{enumerate}
 If $f, g$ be nonnegative
continuous functions   on $\left[0,\infty\right)$  satisfying
$f(t)g(t) = t$, $(t \ge 0)$. Then,
\begin{multline}
w^{2p} \left( AB \right) \le
\frac{1}{2} \left\| AB \right\|^{2p}  + \frac{\gamma}{{2^{p+2} }}\left( {\left\| {B^2 } \right\| + \left\| {B^4 } \right\|^{1/2} } \right)^p \\
\times \left\{ {\left\| {f^{\alpha p} \left( {\left| {A  }
\right|} \right)} \right\| + \left\| {g^{\beta p} \left( {\left|
{A^* } \right|} \right)} \right\| +\Phi\left(f,g;A\right)  }
\right\}
  \label{eq3.6}
\end{multline}
where $\Phi\left(f,g;A\right)$ is defined in Theorem \ref{thm2},
and $\gamma= \max\{\frac{1}{\alpha},\frac{1}{\beta}\}$, for all
$p\ge1$, $\alpha \ge \beta >1 $ with
$\frac{1}{\alpha}+\frac{1}{\beta}=1$ and $\beta p \ge2$.
\end{theorem}

\begin{proof}
Let us first note that the Dragomir refinement  of  Cauchy-Schwarz
inequality  reads that \cite{D5}:
\begin{align*}
\left| {\left\langle {x,y} \right\rangle } \right| \le \left|
{\left\langle {x,e} \right\rangle \left\langle {e,y} \right\rangle
} \right| + \left| {\left\langle {x,y} \right\rangle  -
\left\langle {x,e} \right\rangle \left\langle {e,y} \right\rangle
} \right| \le \left\| x \right\|\left\| y \right\|
\end{align*}
for all $x,y,e\in \mathscr{H}$ with $\|e\|=1$.

It's easy to deduce the inequality
\begin{align}
\left| {\left\langle {x,e} \right\rangle \left\langle {e,y}
\right\rangle } \right| \le \frac{1}{2}\left( {\left|
{\left\langle {x,y} \right\rangle } \right| + \left\| x
\right\|\left\| y \right\|} \right).\label{key}
\end{align}
Setting $e = u$, $x = ABu$, $y = B^*A^*u$ in \eqref{key} and using
the Power-Mean inequality \eqref{PMI} we get
\begin{align*}
\left| {\left\langle {ABu ,u  } \right\rangle \left\langle {u
,B^*A^* u } \right\rangle } \right| &\le \frac{1}{2}\left( {\left|
{\left\langle {ABu ,B^*A^*u } \right\rangle } \right| + \left\|
{Au } \right\|\left\| {B^*A^* u } \right\|} \right)
\\
&\le \left( {\frac{{\left| {\left\langle {(AB)^2 u ,u  }
\right\rangle } \right|^p  + \left\| {Au  } \right\|^p \left\|
{A^* u  } \right\|^p }}{{2^p }}} \right)^{\frac{1}{p}},
\end{align*}
since $A,B$ are commutative so that $(AB)^2=A^2B^2$. Equivalently,
we may write
\begin{align}
  \left| {\left\langle {ABu,u } \right\rangle } \right|^{2p}  \le \frac{1}{2}\left( {\left| {\left\langle {(AB)^2 u ,u } \right\rangle } \right|^p  + \left\| {ABu } \right\|^p \left\| {B^*A^* u } \right\|^p } \right).\label{eq3.8}
\end{align}
Now, using the  mixed Schwarz inequality \eqref{kittaneh.ineq} by
replacing $A,B$ by $A^2,B^2$; respectively, then we have
\begin{align*}
&\left| {\left\langle {(AB)^2x,x} \right\rangle } \right|^p\\ &\le r^p\left(B^2\right)\left\| {f\left( {\left| A^2 \right|} \right)x} \right\|^p\left\| {g\left( {\left| {\left(A^2\right)^* } \right|} \right)x} \right\|^p \\
&=  r^p\left(B^2\right)\left\langle {f^2\left( {\left| {A^2 } \right|} \right)x  ,x  } \right\rangle ^{\frac{p}{2}} \left\langle {g^2\left( {\left| {\left(A^2\right)^*} \right|} \right)x  ,x  } \right\rangle ^{\frac{p}{2}}  \\
&\le r^p\left(B^2\right) \left[\frac{1}{\alpha }\left\langle {f^2 \left( {\left| {A^2 } \right|} \right)x  ,x  } \right\rangle ^{\frac{{\alpha p}}{2}}  + \frac{1}{\beta }\left\langle {g^2 \left( {\left| {\left(A^2\right)^* } \right|} \right)x  ,x  } \right\rangle ^{\frac{{\beta p}}{2}}  \right] \qquad \text{(by \eqref{YI})}\\
&\le  r^p\left(B^2\right)  \left[\frac{1}{\alpha }\left\langle {f^{\alpha p} \left( {\left| {A^2 } \right|} \right)x  ,x  } \right\rangle  + \frac{1}{\beta }\left\langle {g^{\beta p} \left( {\left| {\left(A^2\right)^* } \right|} \right)x  ,x  } \right\rangle \right] \qquad \text{(by \eqref{mc1})}\\
&=  r^p\left(B^2\right)  \left\langle {\left[
{\frac{1}{\alpha}f^{\alpha p} \left( {\left| {A^2  } \right|}
\right) +\frac{1}{\beta }g^{\beta p} \left( {\left|
{\left(A^2\right)^* }\right|} \right)} \right]x  ,x }.
\right\rangle.
\end{align*}
Substituting in \eqref{eq3.8} and taking the supremum over $x \in
\mathscr{H}$, and finally proceed as in the proof of Theorem
\ref{thm2} we get  the desired inequality. We shall omit the
details.
\end{proof}

 \begin{corollary}
Under the assumptions of Theorem \ref{thm3}, we have
\begin{align*}
w^{2} \left( AB \right) \le \frac{1}{2} \left\| AB \right\|^{2}  +
\frac{1}{{8 }}\left( {\left\| {B^2 } \right\| + \left\| {B^4 }
\right\|^{1/2} } \right) \left( {\left\| {A } \right\| + \left\|
{A^2 } \right\|^{1/2} } \right)
\end{align*}
 \end{corollary}

\begin{theorem}
\label{thm4}Let  $A,B\in \mathscr{B}\left( \mathscr{H}\right)^+$
such that $AB$ is contraction.  Then
\begin{multline}
w^{2p}\left(AB\right) \le \left[\left\| {A^p } \right\| - \ell
\left( {\left\| {\left| {\left( {A - \left\| A \right\|} \right)}
\right|^p } \right\|} \right)\right]
\\
\times \left[\left\| {B^p } \right\| - \ell \left( {\left\|
{\left| {\left( {B - \left\| B \right\|} \right)} \right|^p }
\right\|} \right)\right]
\end{multline}
for all $p\ge2$.  In particular, we have
\begin{align*}
w^{2p}\left(A\right)\le  \left\| {A^p } \right\| - \ell \left(
{\left\| {\left| {\left( {A - \left\| A \right\|} \right)}
\right|^p } \right\|} \right)
\end{align*}
for every positive contraction $A$.
\end{theorem}

\begin{proof}
    Let us first prove that $T=\left[ {\begin{array}{*{20}c} A & {B^*A^* }  \\AB & B
    \\\end{array}} \right] \in  \mathscr{B}\left( \mathscr{H}\oplus \mathscr{H}\right)$ is positive.
Since $AB$ is contraction, then by Proposition I.3.5 (\cite{B}, p.
10), $ \left[ {\begin{array}{*{20}c}
    I & {B^* A^* }  \\
    {AB} & I  \\
    \end{array}} \right]$ is positive. Thus,
\begin{align*}
T = \left[ {\begin{array}{*{20}c}
    A & {B^* A^* }  \\
    {AB} & B  \\
    \end{array}} \right] = \left[ {\begin{array}{*{20}c}
    A & 0  \\
    0 & B  \\
    \end{array}} \right] + \left[ {\begin{array}{*{20}c}
    0 & {B^* A^* }  \\
    {AB} & 0  \\
    \end{array}} \right]
\end{align*}
is positive since $A,B\ge0$.

Now, let    ${\bf{x}} = \left( {\begin{array}{*{20}c}
    x_1  \\
    x_2  \\
    \end{array}} \right)
$ in $\mathscr{H} \oplus\mathscr{H} $, such that
$\left\|x_1\right\|^2+\left\|x_2 \right\|^2=1$. Since $AB$ is
contraction then $\left[ {\begin{array}{*{20}c} A & {B^*A^* }
\\AB & B
    \\\end{array}} \right]$  is positive. Therefore, by setting $C=AB$ in Lemma \ref{lemma4}, and this implies that
\begin{align}
\left| {\left\langle
    {ABx_1,x_2}
    \right\rangle } \right|^{2p}  &\le \left\langle{Ax_1,x_1}
\right\rangle^p \left\langle{Bx_2,x_2} \right\rangle^p
\end{align}
If we wish setting $x_1=x_2$ and employing the first inequality in
Lemma \ref{lemma7}, we get
    \begin{align*}
    \left| {\left\langle{ABx_1,x_2} \right\rangle } \right|^{2p}  &\le \left\langle{Ax_1,x_1}\right\rangle^p \left\langle{Bx_2,x_2}\right\rangle^p
    \\
    &\le \left( \left\langle { A^p x_1,x_1} \right\rangle- \left\langle{\left| {A - \left\langle    { A x_1,x_1} \right\rangle 1_{\mathcal{H}} } \right|^px_1,x_1} \right\rangle \right)
    \\
    &\qquad\times\left(\left\langle { B^p x_2,x_2} \right\rangle- \left\langle{\left| {B - \left\langle { B x_2,x_2} \right\rangle 1_{\mathcal{H}} } \right|^px_2,x_2} \right\rangle \right)
    \\
    &\le   \left( \left\langle { A^p x_1,x_1} \right\rangle- \left\langle{\left| {A - \sup_{\|x_1\|=1}\left\langle  { A x_1,x_1} \right\rangle 1_{\mathcal{H}} } \right|^px_1,x_1} \right\rangle \right)
    \\
    &\qquad \times\left(\left\langle { B^p x_2,x_2} \right\rangle- \left\langle{\left| {B - \sup_{\|x_2\|=1}\left\langle    { B x_2,x_2} \right\rangle 1_{\mathcal{H}} } \right|^px_2,x_2} \right\rangle \right)
    \end{align*}
Taking the supremum over $x_1,x_2 \in \mathscr{H}$, we observe
that
\begin{align*}
&\sup_{\|x_1\|=\|x_2\|=1}\left| {\left\langle{ABx_1,x_2}    \right\rangle } \right|^{2p} \\
&\le\sup_{\|x_1\|=\|x_2\|=1} \left\{ \left[\left\langle { A^p
x_1,x_1} \right\rangle- \left\langle{\left| {A -\sup_{\|x_1\|=1}
\left\langle { A x_1,x_1} \right\rangle 1_{\mathcal{H}} }
\right|^px_1,x_1} \right\rangle \right] \right.
\\
&\qquad\left.\times\left[ \left\langle { B^p x_2,x_2}
\right\rangle- \left\langle{\left| {B -
\sup_{\|x_2\|=1}\left\langle  { B x_2,x_2} \right\rangle
1_{\mathcal{H}} } \right|^px_2,x_2} \right\rangle \right] \right\}
\\
&\le \sup_{\|x_1\| =1}    \left\langle { A^p x_1,x_1}
\right\rangle- \inf_{\|x_1\|=1}\left\langle{\left| {A -
\sup_{\|x_1\|=1}\left\langle  { A x_1,x_1} \right\rangle
1_{\mathcal{H}} } \right|^px_1,x_1} \right\rangle
\\
&\qquad \times\sup_{\|x_1\| =1}  \left\langle { B^p x_2,x_2}
\right\rangle-\inf_{\|x_2\|=1} \left\langle{\left| {B -
\sup_{\|x_2\|=1}\left\langle   { B x_2,x_2} \right\rangle
1_{\mathcal{H}} } \right|^px_2,x_2} \right\rangle
 \\
&= \left[\left\| {A^p } \right\| - \ell \left( {\left\| {\left|
{\left( {A - \left\| A \right\|} \right)} \right|^p } \right\|}
\right)\right]\times \left[\left\| {B^p } \right\| - \ell \left(
{\left\| {\left| {\left( {B - \left\| B \right\|} \right)}
\right|^p } \right\|} \right)\right]
\end{align*}
which completes the proof.
\end{proof}

As we have seen, the positivity assumption  of $T=\left[
{\begin{array}{*{20}c}
    A & {B^* A^* }  \\
    {AB} & B  \\
    \end{array}} \right]$ in Theorem \ref{thm4} is essential. A more general case could be obtained for general operators $A,B$ without any contractivity assumptions.

\begin{corollary}
\label{cor4}Let  $A,B\in \mathscr{B}\left( \mathscr{H}\right)^+$
such that $T=\left[ {\begin{array}{*{20}c}
    A & {C^* }  \\
    {C} & B  \\
    \end{array}} \right]$ is positive.  Then
    \begin{multline}
    w^{2p}\left(C\right)\\
    \le \left[\left\| {A^p } \right\| - \ell \left( {\left\| {\left| {\left( {A - \left\| A \right\|} \right)} \right|^p } \right\|} \right)\right]\times \left[\left\| {B^p } \right\| - \ell \left( {\left\| {\left| {\left( {B - \left\| B \right\|} \right)} \right|^p } \right\|} \right)\right]
    \end{multline}
    for all $p\ge2$.  In particular, we have
    \begin{align*}
    w^{2p}\left(C\right)\le \left[\left\| {A^p } \right\| - \ell \left( {\left\| {\left| {\left( {A - \left\| A \right\|} \right)} \right|^p } \right\|} \right)\right].
    \end{align*}
Moreover, in special case for $C=A$, we have
\begin{align*}
w^{2p}\left(A\right)\le \left[\left\| {A^p } \right\| - \ell
\left( {\left\| {\left| {\left( {A - \left\| A \right\|} \right)}
\right|^p } \right\|} \right)\right].
\end{align*}
\end{corollary}
\begin{proof}
The proof follows from Lemma \ref{lemma5} and Theorem \ref{thm4}.
\end{proof}

\begin{corollary}
\label{cor5}    Let  $T\in \mathscr{B}\left( \mathscr{H}\right)$
be any operator. Then
    \begin{multline}
    w^{2p}\left(T\right)
    \le \left[\left\| {T^p } \right\| - \ell \left( {\left\| {\left| {\left( {T - \left\| T \right\|} \right)} \right|^p } \right\|} \right)\right]\\\times \left[\left\| {\left(T^*\right)^p } \right\| - \ell \left( {\left\| {\left| {\left( {T^* - \left\| T \right\|} \right)} \right|^p } \right\|} \right)\right]
    \end{multline}
    for all $p\ge2$.
\end{corollary}
\begin{proof}
Since $\left[ {\begin{array}{*{20}c}
    |T| & {T^* }  \\
    {T} & |T^*|  \\
    \end{array}} \right]\in \mathscr{B}\left(\mathscr{H}\oplus \mathscr{H}\right)$ is positive  (see\cite{FK4}), then the result follows from Corollary \ref{cor4}.
\end{proof}

\section{Refinements of Numerical radius inequalities   for $n \times n$ matrix Operators}\label{sec4}
Several numerical radius type inequalities improving and refining
the inequality   \eqref{eq1.1} have been recently obtained by many
other authors see for example \cite{AF2}, \cite{D2}, \cite{D3},
\cite{FK1}--\cite{FK3}, \cite{OMN}. Among others, three important
facts concerning the  Numerical radius inequalities   of $n \times
n$  Operator matrices are obtained by different authors which are
grouped together, as follows:

    Let  $A=\left[A_{ij}\right]\in \mathscr{B}\left(\bigoplus _{i = 1}^n \mathscr{H}_i\right)$ such that $A_{ij}\in\mathscr{B}\left(\mathscr{H}_j, \mathscr{H}_i\right)$. Then
    \begin{align}
    \label{eq1.6} w\left(A\right)\le
    \left\{ \begin{array}{l}
    \omega \left( {\left[ {t_{ij}^{\left( 1 \right)} } \right]} \right),\qquad {\rm{Hou \,\&\, Du \,\,in}\,\,}\text{\cite{HD}} \\
    \\
    \omega \left( {\left[ {t_{ij}^{\left( 2 \right)} } \right]} \right) ,\qquad {\rm{BaniDomi \,\&\, Kittaneh \,\,in}\,\,} \text{\cite{BF}}\\
    \\
    \omega \left( {\left[ {t_{ij}^{\left( 3 \right)} } \right]} \right),\qquad  {\rm{AbuOmar \,\&\, Kittaneh \,\,in}\,\,} \text{\cite{AF1}}
    \end{array} \right.;
    \end{align}
    where
    \begin{align*}
    t_{ij}^{\left( 1 \right)}  = \omega \left( {\left[ {\left\| {T_{ij} } \right\|} \right]} \right),
\qquad
    t_{ij}^{\left( 2 \right)}  = \left\{ \begin{array}{l}
    \frac{1}{2}\left( {\left\| {T_{ii} } \right\| + \left\| {T_{ii}^2 } \right\|^{1/2} } \right),\,\,\,\,\,\,\,i = j
    \\
    \left\| {T_{ij} } \right\|,\qquad\qquad\qquad\,\,\,\,\,\,\,\,\,\,i \ne j
    \end{array} \right. ,
    \end{align*}
    and
    \begin{align*}
    t_{ij}^{\left( 3 \right)}  = \left\{ \begin{array}{l}
    \omega \left( {T_{ii} } \right),\,\,\,\,\,\,\,i = j
    \\
    \left\| {T_{ij} } \right\|,\,\,\,\,\,\,\,\,\,\, i \ne j
    \end{array} \right..
    \end{align*}

 In the next result we refine the latest bound $t_{ij}^{\left( 3 \right)}$ by adding a third part; which is the  numerical range of the sub-operators  on the  opposite diagonal.

\begin{theorem}
\label{thm5}Let  $A=\left[A_{ij}\right]\in
\mathscr{B}\left(\bigoplus _{i = 1}^n \mathscr{H}_i\right)$ such
that $A_{ij}\in\mathscr{B}\left(\mathscr{H}_j,
\mathscr{H}_i\right)$, and $f,g$ be as in Lemma \ref{lemma5}. Then
\begin{align}
w\left(A\right)\le  w\left( \left[a_{ij}\right] \right),
\end{align}
where
\begin{align*}
a_{ij}=  \left\{ \begin{array}{l}
  w\left( {A_{ij} } \right),\qquad\,\,\,\,\,\,\,\,\,j = i \,\,\,\,\,{\rm{and}}\,\,\,\,\,j \ne n - i + 1
  \\
w \left( {A_{ij } } \right),\qquad\,\,\,\,\,\,\,\,j = n - i + 1
\,\,\,\,\,{\rm{and}}\,\,\,\,\,j \ne i
\\
\left\| {A_{ij} } \right\|, \qquad\,\,\,\,\,\,\,\,\,\,\,j \ne n -
i + 1 \,\,\,\,{\rm{and}}\,\,\,\,\,j \ne i
\end{array} \right..
\end{align*}
\end{theorem}

\begin{proof}
Let $x = \left[ {\begin{array}{*{20}c}
    {x_1 } & {x_2 } &  \cdots  & {x_n }  \\
    \end{array}} \right]^T \in \bigoplus _{i = 1}^n \mathscr{H}_i$ with $\|x\|=1$.  For simplicity setting $k_i=n-i+1$, then we have
    \begin{align*}
    \left| {\left\langle {Ax,x} \right\rangle } \right| &= \left| {\sum\limits_{i,j = 1}^n {\left\langle {A_{ij} x_j ,x_i } \right\rangle } } \right| \\
    &\le \sum\limits_{i,j = 1}^n {\left| {\left\langle {A_{ij} x_j ,x_i } \right\rangle } \right|}  \nonumber\\
    &\le \sum\limits_{i  = 1}^n {\left| {\left\langle {A_{ii} x_i ,x_i } \right\rangle } \right|} + \sum\limits_{i=1}^n {\left| {\left\langle {A_{k_i k_i} x_{k_i} ,x_{k_i} } \right\rangle } \right|}+ \sum\limits_{j \ne i,k_i }^n {\left| {\left\langle {A_{ij} x_j ,x_i } \right\rangle } \right|} \nonumber\\
    &\le  \sum\limits_{i = 1}^n {   \omega \left( {A_{ii} } \right)\left\| {x_i } \right\|^2 } + \sum\limits_{i = 1}^n {    \omega \left( {A_{k_ik_i} } \right)\left\| {x_{k_i} } \right\|^2 } +\sum\limits_{j \ne i}^n {\left\| {A_{ij} } \right\|\left\| {x_i } \right\|\left\| {x_j } \right\|}
    \\
    &\le \sum_{i=1}^n {a_{ij}\left\| {x_i } \right\|\left\| {x_j } \right\|}
    \\
    &= \left\langle {\left[ {a_{ij} } \right]y,y} \right\rangle
    \end{align*}
    where $y=\left( {\begin{array}{*{20}c}{\left\| {x_1 } \right\|} & {\left\| {x_2 } \right\|} &  \cdots  & {\left\| {x_n } \right\|}  \\  \end{array}} \right)^T$. Taking the supremum over $x  \in \bigoplus \mathscr{H}_i$, we obtain the desired result.
\end{proof}

\begin{corollary}
    If $\bf{A}=\left[ {\begin{array}{*{20}c}
        {A_{11} } & {A_{12} }  \\
        {A_{21} } & {A_{22} }  \\
        \end{array}} \right]$ in $ \mathscr{B}\left(\mathscr{H}_1+\mathscr{H}_2\right)$ , then
    \begin{align}
    w\left( \left[ {\begin{array}{*{20}c}
        {A_{11} } & {A_{12} }  \\
        {A_{21} } & {A_{22} }  \\
        \end{array}} \right]\right)
    &\le  \frac{1}{2}\left(w\left( {A_{11} } \right) + w\left( {A_{22} } \right) \right.\\&\qquad\left.+ \sqrt {\left( {w\left( {A_{11} } \right) - w\left( {A_{22} } \right)} \right)^2  + \left( {w\left( {A_{12} } \right) + w\left( {A_{21} } \right)} \right)^2 }   \right)\nonumber
    \end{align}
\end{corollary}
\begin{proof}
From Theorem \ref{thm5}, we have
    \begin{align*}
    w\left( \left[ {\begin{array}{*{20}c}
        {A_{11} } & {A_{12} }  \\
        {A_{21} } & {A_{22} }  \\
        \end{array}} \right]\right) &\le w
    \left( \left[ {\begin{array}{*{20}c}
        \begin{array}{l}
         w\left( {A_{11} } \right)   \\
        \\
        \end{array} & \begin{array}{l}
         w\left( {A_{12} } \right) \\
        \\
        \end{array}  \\
         w\left( {A_{21} } \right)  &  w\left( {A_{22} } \right) \\
        \end{array}}\right] \right)
    \\
    &= \frac{1}{2}r     \left( \left[ {\begin{array}{*{20}c}
        \begin{array}{l}
        w\left( {A_{11} } \right)   \\
        \\
        \end{array} & \begin{array}{l}
        w\left( {A_{12} } \right) +w\left( {A_{21} } \right)\\
        \\
        \end{array}  \\
        w\left( {A_{21} } \right) +w\left( {A_{12} } \right) &  w\left( {A_{22} } \right) \\
        \end{array}}\right] \right)
    \\
    &=  \frac{1}{2}\left(w\left( {A_{11} } \right) + w\left( {A_{22} } \right)\right.
    \\
    &\qquad\left. + \sqrt {\left( {w\left( {A_{11} } \right) - w\left( {A_{22} } \right)} \right)^2  + \left( {w\left( {A_{12} } \right) + w\left( {A_{21} } \right)} \right)^2 }  \right)
    \end{align*}
  which proves the result.
\end{proof}

Using the fact that for any $n\times n$ matrix
$A=\left[A_{ij}\right]$ such that $A_{ij}\ge0$. Then
$w\left(A\right)\le  r\left(
\frac{\left[A_{ij}\right]+\left[A_{ji}\right]}{2} \right)$, we may
state Theorem \ref{thm5} as follows:
\begin{corollary}
\label{cor7}    Let  $A=\left[A_{ij}\right]\in
\mathscr{B}\left(\bigoplus _{i = 1}^n \mathscr{H}_i\right)$ and
$f,g$ be as in Lemma \ref{lemma5}. Then
    \begin{align}
    w\left(A\right)\le  r\left( \left[b_{ij}\right] \right)
    \end{align}
    where
    \begin{align*}
    b_{ij}=  \left\{ \begin{array}{l}
    w\left( {A_{ij} } \right),\qquad\qquad\qquad\qquad \,\,\,\,\,\,\,j = i \,\,\,\,\,{\rm{and}}\,\,\,\,\,j \ne n - i + 1 \\
    \\
    \frac{1}{2}\left(w \left( {A_{ij}}\right)+w \left( {A_{ji}}\right)\right), \qquad \,\,\,\,\,\,j=n-i+1  \,\,\,\,\,{\rm{and}}\,\,\,\,\,j\ne i\\
    \\
    \frac{1}{2}\left(\left\| {A_{ij} } \right\|+\left\| {A_{ji} } \right\|\right), \qquad\,\,\,\,\,\,\,\,\,\,\,\,\,j \ne i, n - i + 1 \\
    \end{array} \right.
    \end{align*}
\end{corollary}

\begin{theorem}
\label{thm6}    Let  $A=\left[A_{ij}\right]\in
\mathscr{B}\left(\bigoplus _{i = 1}^n \mathscr{H}_i\right)$ and
$f,g$ be as in Lemma \ref{lemma5}. Then
    \begin{align}
    w\left(A\right)\le  w\left( \left[c_{ij}\right] \right),
    \end{align}
    where
    \begin{align*}
    c_{ij}= \left\{ \begin{array}{l}
    \frac{1}{2}\left\| {f^2 \left( {\left| {A_{ii} } \right|} \right) + g^2 \left( {\left| {A_{ii}^* } \right|} \right)} \right\|,\qquad\,\,\,\,\,\,\,\,\,j = i \,\,\,\,\,{\rm{and}}\,\,\,\,\,j \ne n - i + 1 \\
    \\
    \frac{1}{2}\left\| {f^2 \left( {\left| {A_{ij } } \right|} \right) + g^2 \left( {\left| {A_{ij}^* } \right|} \right)} \right\|,\,\qquad\,\,\,\, j =n-i+1  \,\,\,\,\,{\rm{and}}\,\,\,\,\,j \ne i  \\
    \\
    \left\| {A_{ij} } \right\|,\qquad\qquad\qquad\qquad\qquad\,\,\,\,\,\,\,\,\,j \ne i, n-i+1 \\  \\
    \end{array} \right..
    \end{align*}

\end{theorem}

\begin{proof}
Let $x = \left[ {\begin{array}{*{20}c}
    {x_1 } & {x_2 } &  \cdots  & {x_n }  \\
    \end{array}} \right]^T \in \bigoplus _{i = 1}^n \mathscr{H}_i$ with $\|x\|=1$.  For simplicity setting $k_i=n-i+1$, then we have
    \begin{align*}
    \left| {\left\langle {Ax,x} \right\rangle } \right| &= \left| {\sum\limits_{i,j = 1}^n {\left\langle {A_{ij} x_j ,x_i } \right\rangle } } \right| \\
    &\le \sum\limits_{i,j = 1}^n {\left| {\left\langle {A_{ij} x_j ,x_i } \right\rangle } \right|}  \nonumber\\
    &\le \sum\limits_{i  = 1}^n {\left| {\left\langle {A_{ii} x_i ,x_i } \right\rangle } \right|} + \sum\limits_{i=1}^n {\left| {\left\langle {A_{k_ik_i} x_{k_i} ,x_{k_i} } \right\rangle } \right|}+ \sum\limits_{j \ne i }^n {\left| {\left\langle {A_{ij} x_j ,x_i } \right\rangle } \right|} \nonumber
\\
    &\le  \sum\limits_{i = 1}^n {\left\langle {f^2 \left( {\left| {A_{ii} } \right|} \right)x_i ,x_i } \right\rangle ^{1/2} \left\langle {g^2 \left( {\left| {A^*_{ii} } \right|} \right)x_i ,x_i } \right\rangle ^{1/2}}
    \\
    &\qquad+\sum\limits_{i = 1}^n {\left\langle {f^2 \left( {\left| {A_{k_ik_i} } \right|} \right)x_{k_i} ,x_{k_i} } \right\rangle ^{1/2} \left\langle {g^2 \left( {\left| {A^*_{k_ik_i} } \right|} \right)x_{k_i} ,x_{k_i} } \right\rangle ^{1/2}}
\\
    &\qquad+ \sum\limits_{j \ne i }^n {\left| {\left\langle {A_{ij} x_j ,x_i } \right\rangle } \right|}
\\
    &\le\frac{1}{2}\left[\sum\limits_{i = 1}^n {\left\langle {f^2 \left( {\left| {A_{ii} } \right|} \right)x_i ,x_i } \right\rangle + \left\langle {g^2 \left( {\left| {A^*_{ii} } \right|} \right)x_i ,x_i } \right\rangle} \right] \\
    &\qquad+\frac{1}{2}\left[\sum\limits_{i = 1}^n {\left\langle {f^2 \left( {\left| {A_{k_ik_i} } \right|} \right)x_{k_i} ,x_{k_i} } \right\rangle + \left\langle {g^2 \left( {\left| {A^*_{k_ik_i} } \right|} \right)x_{k_i} ,x_{k_i} } \right\rangle } \right]
\\
    &\qquad+ \sum\limits_{j \ne i }^n {\left| {\left\langle {A_{ij} x_j ,x_i } \right\rangle } \right|}
          \end{align*}
  \begin{align*}
    &\le\frac{1}{2} \sum\limits_{i = 1}^n {\left\langle {\left[ f^2 \left( {\left| {A_{ii} } \right|} \right)+ g^2 \left( {\left| {A^*_{ii} } \right|} \right)\right]x_i ,x_i } \right\rangle}  \\
    &\qquad+\frac{1}{2}\sum\limits_{i = 1}^n {\left\langle {\left[ f^2 \left( {\left| {A_{k_ik_i} } \right|} \right)+g^2 \left( {\left| {A^*_{k_ik_i} } \right|} \right)\right] x_{k_i} ,x_{k_i} } \right\rangle   }
    \\
    &\qquad+ \sum\limits_{j \ne i }^n {\left| {\left\langle {A_{ij} x_j ,x_i } \right\rangle } \right|}
    \\
    &= \left\langle {\left[ {c_{ij} } \right]y,y} \right\rangle
    \end{align*}
    where $y=\left( {\begin{array}{*{20}c}{\left\| {x_1 } \right\|} & {\left\| {x_2 } \right\|} &  \cdots  & {\left\| {x_n } \right\|}  \\  \end{array}} \right)^T$. Taking the supremum over $x  \in \bigoplus \mathscr{H}_i$, we obtain the desired result.
\end{proof}

As we did in Corollary \ref{cor7} we may restate Theorem
\ref{thm6} in terms of spectral radius as follows:
\begin{corollary}
    Let  $A=\left[A_{ij}\right]\in \mathscr{B}\left(\bigoplus _{i = 1}^n \mathscr{H}_i\right)$ and $f,g$ be as in Lemma \ref{lemma5}. Then
    \begin{align}
    w\left(A\right)\le  r\left( \left[d_{ij}\right] \right),
    \end{align}
    where
    \begin{align*}
    d_{ij}= \left\{ \begin{array}{l}
    \frac{1}{2}\left\| {f^2 \left( {\left| {A_{ij} } \right|} \right) + g^2 \left( {\left| {A^*_{ij} } \right|} \right)} \right\|,  \qquad\qquad\qquad\,\,\,\,\,\,\,\,\,j = i \,\,\,\,\,{\rm{and}}\,\,\,\,\,j \ne n - i + 1 \\
    \\
    \frac{1}{4}\left[\left\| {f^2 \left( {\left| {A_{ij } } \right|} \right) + g^2 \left( {\left| {A^*_{ij} } \right|} \right)} \right\|\right.\\ \qquad\qquad\left.+\left\| {f^2 \left( {\left| {A_{ji } } \right|} \right) + g^2 \left( {\left| {A^*_{ji} } \right|} \right)} \right\|\right], \qquad\,\,\,\, j =n-i+1  \,\,\,\,\,{\rm{and}}\,\,\,\,\,j \ne i  \\
    \\
    \frac{1}{2}\left(\left\| {A_{ij} } \right\|+\left\| {A_{ji} } \right\|\right),   \qquad\qquad\qquad\qquad\qquad\,\,\,\,\,j \ne i, n-i+1 \\  \\
    \end{array} \right..
    \end{align*}

\end{corollary}

\centerline{}\centerline{}

 \textbf{Acknowledgment:} The author wish to thank the referee for his fruitful comments and careful reading of the original manuscript of this work that have implemented the final version of this work.

\end{document}